\documentclass[11pt,a4paper,reqno,twoside]{article}
\usepackage{hyperref}
\usepackage{amsfonts}
\usepackage{bbding}
\usepackage[all]{xy}
\usepackage{young}
\usepackage{multirow}
\usepackage{booktabs}
\usepackage{amsmath,amscd,amssymb,latexsym,srcltx,indentfirst,titlesec}
\usepackage{amsthm}

\usepackage{enumitem}

\setitemize[1]{itemsep=0pt,partopsep=0pt,parsep=\parskip,topsep=5pt}

\numberwithin{equation}{section}
\newtheorem{theorem}{Theorem}[section]
\newtheorem{lemma}[theorem]{Lemma}
\newtheorem{proposition}[theorem]{Proposition}

\newtheorem{remark}[theorem]{Remark}

\newtheorem{conjecture}[theorem]{Conjecture}
\newtheorem{definition}[theorem]{Definition}

\newtheorem*{thm*}{Main Theorem}
\allowdisplaybreaks

\begin{document}

\makeatletter

\newdimen\bibspace
\setlength\bibspace{2pt}
\renewenvironment{thebibliography}[1]{%
 \section*{\refname %or \bibname if you use ``book'' as the documentclass
       \@mkboth{\MakeUppercase\refname}{\MakeUppercase\refname}}%
     \list{\@biblabel{\@arabic\c@enumiv}}%
          {\settowidth\labelwidth{\@biblabel{#1}}%
           \leftmargin\labelwidth
           \advance\leftmargin\labelsep
           \itemsep\bibspace
           \parsep\z@skip     %
           \@openbib@code
           \usecounter{enumiv}%
           \let\p@enumiv\@empty
           \renewcommand\theenumiv{\@arabic\c@enumiv}}%
     \sloppy\clubpenalty4000\widowpenalty4000%
     \sfcode`\.\@m}
    {\def\@noitemerr
      {\@latex@warning{Empty `thebibliography' environment}}%
     \endlist}

\makeatother

\pdfbookmark[2]{ A weak form of a conjecture on vanishing of group cohomology for  blocks}{beg}

\title{{\bf  A weak form of a conjecture on vanishing of group cohomology for  blocks} }

\footnotetext{~\\
1. Department of Mathematics, Hubei Key Laboratory of Applied Mathematics, Hubei University, Wuhan, 430062, China\\
2. School of Mathematical Sciences, Peking University, Beijing, 100871, China\\
\hspace*{2 ex}Heguo Liu's E-mail: ghliu@hubu.edu.cn

Xingzhong Xu's E-mail: xuxingzhong407@hubu.edu.cn; xuxingzhong407@126.com

Jiping Zhang's E-mail:  jzhang@pku.edu.cn}

\author{{\small{Heguo Liu$^{1}$, Xingzhong Xu$^{1}$, Jiping Zhang$^{2}$} }
}

\date{Feb/15/2021}
\maketitle

{\small

\noindent\textbf{Abstract.} {\small{ In this paper, we get an elementary and important lemma(See Lemma 3.2) which is about pushout and pullback of modules. 
And we prove a weak form of a long open conjecture on vanishing of group cohomology for blocks.}}\\


\noindent\textbf{Mathematics Subject Classification (2010):}  }

\section{\bf Introduction}

Let $G$ be a finite group and $k$ be a field. In this paper,  all $kG$-module are finite generated. We all know that:
If $V, W$ are simple  simple $kG$-modules, then $V, W$ are isomorphic to each other if and only if
$$\mathrm{Hom}_{kG}(V, W)\neq 0.$$

What's about the case that $V, W$ are just in the same block? There is a long conjecture as follows:

\begin{conjecture}\cite[Conjecture 1.1]{Li}  Let $G$ be a finite group, $p$ a prime, $k$ a field of characteristic $p$.
If $V$ is a simple $kG$-mod in the principal block of $kG$, then
there exists $i\in \mathbb{N}$ such that
$$H^i(G, V)=\mathrm{Ext}^i_{kG}(k, V)\neq 0.$$
\end{conjecture}

We all know that $\mathrm{Ext}^i_{kG}(V, W)\cong \underline{\mathrm{Hom}}_{kG}(\Omega^n(V), W)$.
Here, $\underline{\mathrm{Hom}}_{kG}(\Omega^n(V), W)$ is the morphism set of the stable module category $\mathrm{Stmod}(kG)$.
$\Omega^n(V)$ means the n$^{th}$ Heller of $V$.

There are a few partial results for this conjecture. In \cite{Li, LiS1}, Linnell and Stammbach proved that this conjecture holds for $p$-soluble groups. In \cite{LiS2}, Linnell and Stammbach proved that this conjecture holds for $p$-constrained groups.
In \cite{BCR}, Benson, Carlson and Robinson discussed this conjecture and gave a equivalent condition to this conjecture when $V$ is just a non-projective module. In page 42 of \cite{BCR},
they said this question again.

In this paper, we prove the following main theorem:

\begin{thm*} [Weak form of the Conjecture 1.1] Let $G$ be a finite group, $p$ a prime, $k$ a field of characteristic $p$.
If $T$ is a simple $kG$-modules in the principal block of $kG$,
then there is a list of $kG$-modules $k=V_1, V_2,\ldots, V_n=T$
such that for each $i=1,\ldots, n-1,$ the modules $V_i $ and $ V_{i+1}$
satisfy that
$$\mathrm{Ext}_{kG}^{f(i)}(V_{i}, V_{i+1})\neq 0$$
where $f(i)$ are positive integers for $i=1,2,\ldots, n-1$.
\end{thm*}

If we can choose that $n=2$, the Main theorem implies the Conjecture 1.1 holds. So the Main theorem is a weak form of the Conjecture 1.1.
Here, we need to remark that the above theorem and the Conjecture 1.1 are not true for non-principal block because there exist  simple and projective $kG$-modules who belong to some non-principal blocks.

$Structure~ of ~ the~ paper:$
After recalling the basic definitions and properties of modules and blocks  in the Section 2,
we prove two lemmas about exact sequences in the Section 3, and we lists some properties of $\mathrm{Ext}$ functor in the Section 4.
In the Section 5, we prove the Main theorem by induction.

\section{\bf Some properties of blocks}

In this section we collect some known results about blocks, we refer to \cite{W}.

We make the definition that a block of a ring $A$ with identity is a primitive
idempotent in the center $Z(A)$.

We will say that an $A$-module $U$(left module) belongs to (or lies in) the block $e$ if $eU = U$.

\begin{definition} We put an equivalence relation on the set of simple modules of an algebra $A$:
define $S\sim T$ if either $S\cong T$ or there is a list of simple $A$-modules $S=S_1, S_2,\ldots, S_n=T$
 so that for each $i=1,\ldots, n-1,$ the modules $S_i $ and $ S_{i+1}$
appear in a non-split short exact sequence of $A$-modules
$0\to U\to V\to  W\to 0 $  with $\{U,W\}=\{S_i, S_{i+1}\}$.
\end{definition}

\begin{proposition}\cite[Proposition 12.1.7]{W} Let $A$ be a finite-dimensional algebra over a field $k$. The
following are equivalent for simple $A$-modules $S$ and $T$:

(1) $S$ and $T$ lie in the same block.

(2) There is a list of simple A-modules $S=S_1, S_2,\ldots, S_n=T$ so that $S_i$
and $S_{i+1}$ are both composition factors of the same indecomposable projective module, for each $i=1,\ldots, n-1.$

(3) $S\sim T$.
\end{proposition}

\begin{proof} See \cite[Proposition 12.1.7]{W}.
\end{proof}

Let $M$ be a $kG$-module, recall that $k$-dual $M^{\ast}=\mathrm{Hom}(M, k)$ as the $kG$-module of the $k$-linear homomorphisms from
$M$ to the trivial module $k$. Then $G$-action reads
$$(gf)(m)=f(g^{-1}m)~~~~\quad\quad(for~g\in G, ~f\in M^{\ast},~ m\in M).$$

\begin{remark}Let $T$ be a simple $kG$-module in the principal block of $kG$, then
 $T^{\ast}$ is also a simple $kG$-module in the principal block of $kG$.
\end{remark}

\begin{proof} Since the trivial module $k$ is in the principal block of $kG$.
By above proposition, we have $k\sim T$. That is
either $S\cong T$ or there is a list of simple $A$-modules $k=S_1, S_2,\ldots, S_n=T$
 so that for each $i=1,\ldots, n-1,$ the modules $S_i $ and $ S_{i+1}$
appear in a non-split short exact sequence of $A$-modules
$0\to U\to V\to  W\to 0 $  with $\{U,W\}=\{S_i, S_{i+1}\}$.
So, we have a list $k^{\ast}=S_1^{\ast}, S_2^{\ast},\ldots, S_n^{\ast}=T^{\ast}$. And for each $i=1,\ldots, n-1,$ the modules $S_i $ and $ S_{i+1}$
appear in a non-split short exact sequence of $A$-modules
$0\to U\to V\to  W\to 0 $  with $\{U,W\}=\{S_i, S_{i+1}\}$. That means
$S_i^{\ast} $ and $ S_{i+1}^{\ast}$
appear in a non-split short exact sequence of $A$-modules
$0\to W^{\ast}\to V^{\ast}\to  U^{\ast}\to 0 $  with $\{U^{\ast},W^{\ast}\}=\{S_i^{\ast}, S_{i+1}^{\ast}\}$.
But $k^{\ast}\cong k$, we have $k\sim T^{\ast}$. That is $T^{\ast}$ is in the principal block of $kG$.
\end{proof}

\section{\bf Some properties of exact sequences}

In this section, we list two lemmas which will be used to prove the main result.
The Lemma 3.2 is constructed by us and the Lemma 3.1 is well-known.
\begin{lemma}Let $$\CD
 0 @> >> U @>\alpha >> V_n @> \beta  >> W @> >> 0
\endCD$$ be an exact sequence. Let
$$\CD
 0 @> >> W @>\gamma_n >> V_{n-1} @> \gamma_{n-1} >> \cdots @>\gamma_2 >> V_1 @> \gamma_1  >> T @> >> 0
\endCD$$ be a long exact sequence. Then
we have a long exact sequence
$$\CD
 0 @> >> U @>\alpha >> V_{n} @> \gamma_{n}\beta >> V_{n-1} @> \gamma_{n-1} >> \cdots  @>\gamma_2 >> V_1 @> \gamma_1  >> T @> >> 0
\endCD$$
\end{lemma}

\begin{proof} First, $\gamma_{n-1}\gamma_{n}\beta=0$, we have $\mathrm{Im}(\gamma_{n}\beta)\subseteq \mathrm{Ker}(\gamma_{n-1})$.
Let $x\in \mathrm{Ker}(\gamma_{n-1})=\mathrm{Im}(\gamma_n)$, there exists $y\in W$ such that $\gamma_n(y)=x$. Also $\beta$ is surjective,
there exists $z\in V_n$ such that $\beta(z)=y$. Hence $\gamma_n(\beta(z))=x$. That means $x\in \mathrm{Im}(\gamma_{n}\beta)$. Hence,
$\mathrm{Im}(\gamma_{n}\beta)=\mathrm{Ker}(\gamma_{n-1})$.

Second, $\gamma_{n}\beta\alpha =0$, we have $\mathrm{Im}(\alpha)\subseteq \mathrm{Ker}(\gamma_{n}\beta)$.
Let $x'\in \mathrm{Ker}(\gamma_{n}\beta)$, we have $\gamma_{n}\beta(x')=0.$ But $\gamma_n$ is injective, we have $\beta(x')=0$.
That means $x'\in \mathrm{Ker}(\beta)= \mathrm{Im}(\alpha).$ Hence, $\mathrm{Im}(\alpha)= \mathrm{Ker}(\gamma_{n}\beta)$.
\end{proof}

\begin{lemma} Let $G$ be a finite group, $p$ a prime, $k$ a field of characteristic $p$.
Let $S, S_1, S_2, T, V_0, V_1$ and $V_2$ be $kG$-modules. And there exist short exact sequences as follows:
$$\CD
 0 @> >> S_1 @>\alpha_0 >> V_0 @> \beta_0  >> S @> >> 0
\endCD$$
$$\CD
 0 @> >> S_1 @>\alpha_1 >> V_1 @> \beta_1  >> S_2 @> >> 0
\endCD$$
$$\CD
 0 @> >> T @>\alpha_2 >> V_2 @> \beta_2  >> S_1 @> >> 0.
\endCD$$
Then there exists a long exact sequence as follows:
$$\CD
 0 @> >> T @>  >> W_2 @>   >> W_1 @>    >> S @> >> 0
\endCD$$
where $W_1$ is the pushout of $(S_1, V_0, V_1)$ and $W_2$ is the pullback of $(S_2, V_1, V_2)$.
\end{lemma}

\begin{proof} Since $$\CD
 0 @> >> S_1 @>\alpha_0 >> V_0 @> \beta_0  >> S @> >> 0
\endCD$$
$$\CD
 0 @> >> S_1 @>\alpha_1 >> V_1 @> \beta_1  >> S_2 @> >> 0
\endCD$$
$$\CD
 0 @> >> T @>\alpha_2 >> V_2 @> \beta_2  >> S_1 @> >> 0.
\endCD,$$ we have
$$\xymatrix{
  ~ & ~ & ~ & 0 \ar[d]^{~}  & ~ &  ~ &  ~ \\
~ & ~ & 0 \ar[r]^{~} & S_1\ar[d]^{\alpha_1} \ar[r]^{\alpha_0} & V_0 \ar[d]^{\tau_0} \ar[r]^{\beta_0} & S \ar[r]^{~} & 0  \\
  ~ & ~ & W_2 \ar[d]^{\sigma_2} \ar[r]^{\sigma_1} & V_1\ar[d]^{\beta_1} \ar[r]^{\tau_1}  & W_1 &  ~ &  ~ \\
0\ar[r]^{~} & T \ar[r]^{\alpha_2} & V_2  \ar[r]^{\beta_2} & S_2 \ar[d]^{~} \ar[r]^{~} & 0 & ~ & ~  \\
  ~ & ~ & ~ & 0    & ~ &  ~ &  ~
} $$
$$\mathbf{Diagram ~3.1.}$$
Here, $W_1$ is the pushout and $W_2$ is the pullback.

{\bf Construct $W_1$:} set $U=\{(\alpha_0(t), -\alpha_1(t))\in V_0\oplus V_1|t\in S_1\}$, we define $W_1:=(V_0\oplus V_1)/U.$ And
we have some morphisms as follows:
$$\tau_0: V_0\to W_1~ \mathrm{by}~\tau_0(x)=(x, 0)+U,$$
$$\tau_1: V_1\to W_1~ \mathrm{by}~\tau_1(y)=(0, y)+U.$$
And $\tau_1$ is injective because $\alpha_0$ is injective.
Set $$\beta_0': W_1\to S ~\mathrm{by}~ \beta_0'((x,y)+U)=\beta_0(x).$$
It is easy to see that $\beta_0'$ is well-defined and $\beta_0'$
is surjective. Now, we assert that $\mathrm{Ker}(\beta_0')=\mathrm{Im}(\tau_1).$ It is easy to see that $\mathrm{Ker}(\beta_0')\supseteq\mathrm{Im}(\tau_1).$
Let $(x, y)+U\in \mathrm{Ker}(\beta_0')$, we have $\beta_0(x)=0$. So there exists $t\in S_1$ such that $\alpha_0(t)=x$.
That means $(x, y)+U=(\alpha_0(t), y)+U=(0, y+\alpha_1(t))+U+(\alpha_0(t), -\alpha_1(t))+U=(0, y+\alpha_1(t))+U\in \mathrm{Im}(\tau_1)$.
Hence, $\mathrm{Ker}(\beta_0')=\mathrm{Im}(\tau_1)$. Also $\mathrm{Im}(\tau_1)\cong V_1$ because $\tau_1$ is injective.
So we have a short exact sequence as follows
$$\CD
 0 @> >> V_1 @>\tau_1 >> W_1 @> \beta'_0  >> S @> >> 0
\endCD.$$

{\bf Construct $W_2$:} we can set $W_2=\{(y, z)\in V_1\oplus V_2|\beta_1(y)=\beta_2(z)\}$.  And
we have some morphisms as follows:
$$\sigma_1: W_2\to V_1 ~\mathrm{by}~ \sigma_1(y, z)=y,$$
$$\sigma_2: W_2\to V_2 ~\mathrm{by}~ \sigma_2(y, z)=z.$$
Now, we assert that $\mathrm{Ker}(\sigma_1)\cong T$.
Since $\mathrm{Ker}(\sigma_1)=\{(y, z)\in W_2|\sigma_1(y, z)=0\}=\{(0, z)|(0, z)\in W_2\}=\{(0, z)|\beta_1(0)=\beta_2(z)\}=\{(0, z)|\beta_2(z)=0\}\cong \mathrm{Im}(\alpha_2)\cong T$.
Set $\alpha'_2: T\to W_2$ by $\alpha'_2(t)=(0,\alpha_2(t))$ for $t\in T$. So we have a short exact sequence as follows
$$\CD
 0 @> >> T @>\alpha'_2 >> W_2 @> \sigma_1  >> V_1 @> >> 0
\endCD.$$

Now, we have a long exact sequence as follows
$$\CD
 0 @> >> T @>\alpha'_2 >> W_2 @> \tau_1\sigma_1  >> W_1 @> \beta'_0  >> S @> >> 0
\endCD$$
by Lemma 3.1.
\end{proof}

\begin{remark}(1). If the exact sequences
$$\CD
 0 @> >> S_1 @>\alpha_0 >> V_0 @> \beta_0  >> S @> >> 0
\endCD$$
is non-split, then exact sequences $$\CD
 0 @> >> V_1 @>\tau_1 >> W_1 @> \beta'_0  >> S @> >> 0
\endCD$$
is non-split.

(2). If the exact sequence
$$\CD
 0 @> >> T @>\alpha_2 >> V_2 @> \beta_2  >> S_1 @> >> 0
\endCD$$
is non-split, $S_1, T$ are simple $kG$-modules and $S_1\ncong T$, then
the exact sequence
$$\CD
 0 @> >> T @>\alpha'_2 >> W_2 @> \sigma_1  >> V_1 @> >> 0
\endCD$$
is non-split.
\end{remark}

\begin{proof} (1). Suppose that the exact sequence $$\CD
 0 @> >> V_1 @>\tau_1 >> W_1 @> \beta'_0  >> S @> >> 0
\endCD$$ is split, then there exists a map $\rho$ from $W_1$ to $V_1$ such that
$$\rho\tau_1= \mathrm{Id}_{V_1}.$$

Let $t\in S_1$, we have
$$\tau_0\alpha_0(t)=(\alpha_0(t),0)+U=(0, \alpha_1(t))+(\alpha_0(t), -\alpha_1(t))+U=(0, \alpha_1(t))+U.$$
We can see that $\rho\tau_1(\alpha_1(t))=\rho((0, \alpha_1(t))+U)=\alpha_1(t)$.
So $\rho\tau_0\alpha_0(t)=\alpha_1(t)$. Set $\widetilde{\alpha}_1^{-1}$ is the inverted map of $\alpha_1: S_1\to \mathrm{Im}(\alpha_1)$.
Then we have $\widetilde{\alpha}_1^{-1}\rho\tau_0\alpha_0=\mathrm{Id}_{S_1}$. Hence, the exact sequence
$$\CD
 0 @> >> S_1 @>\alpha_0 >> V_0 @> \beta_0  >> S @> >> 0
\endCD$$
is split, that is a contradiction. Therefore, the exact sequence $$\CD
 0 @> >> V_1 @>\tau_1 >> W_1 @> \beta'_0  >> S @> >> 0
\endCD$$ is non-split.

(2). suppose that the exact sequence $$\CD
 0 @> >> T @>\alpha'_2 >> W_2 @> \sigma_1  >> V_1 @> >> 0
\endCD$$ is split, then there exists a map $\rho$ from $V_1$ to $W_2$ such that
$$\delta\alpha_2'= \mathrm{Id}_{T}.$$

Let $z\in V_2$, set a map $\varrho$ from $V_2$ to $T$ as follows:
$$\varrho: V_2\longrightarrow T$$
$$~~~~z\longmapsto \delta(y, z)$$
for some $y\in V_1$ such that $(y, z)\in W_2$.
Assume that $y, y'\in V_1$ such that $(y, z), (y',z)\in W_2$. So $(y-y', 0)\in W_2$, that means
$y-y'\in \mathrm{Ker}(\beta_1)=\mathrm{Im}(\alpha_1)\cong S_1$.
If $\delta(y-y', 0)\neq 0$, that is a contradiction because $S_1, T$ are simple $kG$-modules and $S_1\ncong T$.
Hence $\delta(y-y', 0)=0.$
So, $\varrho$ is well-defined.
Let $t\in T$, we can see that $\varrho\alpha_2(t)=\delta(0, \alpha_2(t))$ because $\beta_2\alpha_2(t)=0$.
Hence, we have $\varrho\alpha_2(t)=\delta(0, \alpha_2(t))=\delta\alpha_2'(t)=t$. Therefore, we have
$\varrho\alpha_2= \mathrm{Id}_{T}$. That means
 the exact sequence
$$\CD
 0 @> >> S_1 @>\alpha_0 >> V_0 @> \beta_0  >> S @> >> 0
\endCD$$
is split, that is a contradiction.
Therefore, the exact sequence $$\CD
 0 @> >> T @>\alpha'_2 >> W_2 @> \sigma_1  >> V_1 @> >> 0
\endCD$$
is also non-split.
\end{proof}

\section{\bf Some properties of Ext functor}

In this section we collect some known results about Ext functor, we refer to \cite{C, CR, HS, Wei}.

Let $U, W$ be $kG$-modules, and $n$ a non-negative integer. An exact sequence
$$E: 0\to U\to B_{n-1}\to \cdots\to B_0\to W\to 0$$
of $kG$-modules is called an $n$-$extension$ of $W$ by $U$.
Recall that $n$-extensions $E, E'$ satisfy the relation $E\rightarrow E'$ if there is a commutative diagram
$$\CD
E:~~~0 @>  >> U @>  >> B_{n-1} @>   >> \cdots @>   >> B_0 @>   >> W @> >> 0 \\
       @V   VV @V   VV @V  \mu_{n-1} VV @V    VV @V \mu_0 VV @V = VV   \\
E':~~~0 @>  >> U @> >> C_{n-1} @> >> \cdots  @>   >> C_0 @>    >> W @> >> 0.
\endCD $$
It is easy to see that the relation $\rightarrow$ is not symmetric for $n\geq 2$. But we can define $E$ and
$E'$ to be equivalent, $E\sim E'$, if and only if there exists a chains $E=E_0, E_1,\ldots, E_k=E'$ with
$$E_0\rightarrow E_1\leftarrow E_2\rightarrow\cdots\leftarrow E_k.$$
By $[E]$ we denote the equivalence class of the $n$-extension $E$.
Let $\mathcal{U}^n(W, U)$ be the class of all $n$-extension of $W$ by $U$.

\begin{theorem}\cite[Theorem 6.3]{C} \cite[Theorem 9.1]{HS} Let $U, W$ be $kG$-modules, and $n$ a non-negative integer.
Then there is a bijection $\Phi: $
$$\mathrm{Ext}^n_{kG}(W, U)\cong\mathcal{U}^n(W, U)/\sim.$$
\end{theorem}

\begin{proof} See \cite[Theorem 6.3]{C}.
\end{proof}

\begin{lemma}\cite[Lemma 3.4.1]{Wei} Let $U, W, V$ be $kG$-modules. Let $0\to U\to V\to  W\to 0 $
be an exact sequence.
If $\mathrm{Ext}^1_{kG}(W, U)=0$, then the exact sequence $0\to U\to V\to  W\to 0 $
is split.
\end{lemma}

\begin{proof} Set the exact sequence as follows:
$$\CD
 0 @> >> U @>\alpha >> V @> \beta  >> W @> >> 0
\endCD$$

By \cite[Theorem 8.6]{CR}, we have a long exact sequence as follows:
$$\CD
 0 @> >> \mathrm{Hom}_{kG}(W, U) @>\circ\beta>> \mathrm{Hom}_{kG}(V, U)  @> \circ\alpha  >>
\endCD$$ $$
\CD
\mathrm{Hom}_{kG}(U, U)  @> >> \mathrm{Ext}^1_{kG}(W, U)  @> >> \cdots
\endCD
$$
Since $\mathrm{Ext}^1_{kG}(W, U)=0$,
we have an exact sequence as follow:
$$\CD
 0 @> >> \mathrm{Hom}_{kG}(W, U) @>\circ\beta>> \mathrm{Hom}_{kG}(V, U)  @> \circ\alpha  >> \mathrm{Hom}_{kG}(U, U)  @> >>  0
\endCD$$
That means $\mathrm{Id}_U= \phi \circ\alpha$ for some $\phi\in \mathrm{Hom}_{kG}(V, U)$. Hence,
$0\to U\to V\to  W\to 0 $
is split.
\end{proof}

%\begin{lemma}Let $U, W$ be simple $kG$-modules.
%If $\mathrm{Ext}^1_{kG}(U, W)\neq 0$, then
%$$\mathrm{Ext}^1_{kG}(W, U)\neq 0??????$$
%\end{lemma}
%
%\begin{proof} $\mathrm{Ext}^1_{kG}(U, W)\otimes \mathrm{Ext}^1_{kG}(W, U)\to \mathrm{Ext}^2_{kG}(U, U) \neq 0$
%\end{proof}

\begin{lemma}\cite[Lemma 5.2.3]{Ben1} Suppose $k$ is a field of characteristic $p$ an $M$ is a finitely
generated $kG$-module. Then the following are equivalent:

(i) $M$ is projective.

(ii) $\mathrm{Ext}^n_{kG}(M, M)=0$ for all $n> 0$.

(iii) $\mathrm{Ext}^n_{kG}(M, M)=0$ for all $n$ large enough.

(iv) Every element of $\mathrm{Ext}^{\ast}_{kG}(M, M)$ of positive degree is nilpotent.
\end{lemma}

\section{\bf A proof of the Main Theorem}

In this section, we give a proof of the main theorem.
\begin{theorem}  Let $G$ be a finite group, $p$ a prime, $k$ a field of characteristic $p$.
If $T$ is a simple $kG$-modules in the principal block of $kG$,
then there is a list of $kG$-modules $k=V_1, V_2,\ldots, V_n=T$
such that for each $i=1,\ldots, n-1,$ the modules $V_i $ and $ V_{i+1}$
satisfy that
$$\mathrm{Ext}_{kG}^{f(i)}(V_{i}, V_{i+1})\neq 0$$
where $f(i)$ are positive integers for $i=1,2,\ldots, n-1$.
\end{theorem}

\begin{proof} {\bf Step 1.} Since $k, T$ are simple $kG$-modules in the same block of $kG$, there is a list of simple $A$-modules $k=S_0, S_1,\ldots, S_m=T$
 so that for each $i=0,\ldots, m-1,$ the modules $S_i $ and $ S_{i+1}$
appear in a non-split short exact sequence of $A$-modules
$0\to U\to V\to  W\to 0 $  with $\{U,W\}=\{S_i, S_{i+1}\}$ by Proposition 2.2.
By Lemma 4.2, we have $\mathrm{Ext}_{kG}^1(k, S_1)\neq 0$ or $\mathrm{Ext}_{kG}^1(S_1, k)\neq 0$.

If $\mathrm{Ext}_{kG}^1(S_1, k)\neq 0$, we know that $\mathrm{Ext}_{kG}^1(k^{\ast}, S_1^{\ast})\cong\mathrm{Ext}_{kG}^1(S_1, k)\neq 0$.
Since $S^{\ast}_1$ also belongs to the principal block of $kG$,  there is a list of simple $A$-modules $S^{\ast}_1=S_0', S_1',\ldots, S_{m'}'=T$
 so that for each $i=0,\ldots, m'-1,$ the modules $S_i' $ and $ S_{i+1}'$
appear in a non-split short exact sequence of $A$-modules
$0\to U\to V\to  W\to 0 $  with $\{U,W\}=\{S_i', S_{i+1}'\}$ by Proposition 2.2.

So we always choose a list of simple $A$-modules $k:=S:=S_0, S_1,\ldots, S_n=T$
 so that for each $i=0,\ldots, n-1,$ the modules $S_i $ and $ S_{i+1}$
appear in a non-split short exact sequence of $A$-modules
$0\to U\to L\to  W\to 0 $  with $\{U,W\}=\{S_i, S_{i+1}\}$ and
$$\mathrm{Ext}_{kG}^1(k, S_1)\neq 0.$$

{\bf Step 2.} Induction on $n$.

{\bf Step 2.1.} If $n=1$, by our choice, we have $\mathrm{Ext}_{kG}^1(k, S_1)\neq 0$.

{\bf Step 2.2.}
Now, suppose that the theorem holds for $n-1$. That means there is a list of $kG$-modules $k=V_1, V_2,\ldots, V_{r}=S_{n-1}$
such that for each $i=1,\ldots, n-1,$ the modules $V_i $ and $ V_{i+1}$
satisfy that
$$\mathrm{Ext}_{kG}^{f(i)}(V_{i}, V_{i+1})\neq 0$$
where $f(i)$ are positive integers for $i=1,2,\ldots, r-1$.

We know that the modules $S_{n-1}$ and $ S_{n}=T$
appear in a non-split short exact sequence of $A$-modules
$0\to U\to V\to  W\to 0 $  with $\{U,W\}=\{S_{n-1}, S_{n}\}$. So, we will discuss two cases as follows:

{\bf Case 1:} $T$ and $S_{n-1}$ are in the following non-split exact sequence: $$\CD
 0 @> >> T @>\alpha_1 >> L @> \beta_1  >> S_{n-1} @> >> 0.
\endCD$$
That means $$\mathrm{Ext}_{kG}^{1}(S_{n-1}, T)\neq 0.$$
Set $V_{r+1}=T$ and $f(r)=1$, we have
$$\mathrm{Ext}_{kG}^{f(i)}(V_{i}, V_{i+1})\neq 0$$
where $f(i)$ are positive integers for $i=1,2,\ldots, r$.
So we complete the Case 1.

{\bf Case 2:} $T$ and $S_{n-1}$ are in the following non-split exact sequence: $$\CD
 0 @> >> S_{n-1} @>\alpha_1 >> L @> \beta_1  >> T @> >> 0.
\endCD,$$

Since $T$ is not projective, we have $\mathrm{Ext}^{s}_{kG}(T, T)\neq0$ for some positive integer $s$ by Lemma 4.3. By Theorem 4.1, we have the following exact sequence: $$\CD
 0 @> >> T @> \kappa_s >> M_{s-1} @> \kappa_{s-1}  >> \cdots @> \kappa_1 >> M_0 @> \kappa_0 >> T @> >> 0
\endCD $$
$$\mathbf{Diagram ~5.1.}$$
who's image is not 0 under the map $\Phi$.

{\bf Case 2.1.} $\mathrm{Ker}(\kappa_0) \cong S_{n-1}$. We can see that the exact sequence: $$\CD
 0 @> >> T @> \kappa_s >> M_{s-1} @> \kappa_{s-1}  >> \cdots @> \kappa_2 >> M_1 @> \kappa_1 >> \mathrm{Im}(\kappa_1)(\cong\mathrm{Ker}(\kappa_0) \cong S_{n-1}) @> >> 0
\endCD $$ who's image is not 0 under the map $\Phi$.
Else, the exact sequence: $$\CD
 0 @> >> T @> \kappa_s >> M_{s-1} @> \kappa_{s-1}  >> \cdots @> \kappa_1 >> M_0 @> \kappa_0 >> T @> >> 0
\endCD $$
who's image is 0 under the map $\Phi$.
That means $\mathrm{Ext}^{s-1}_{kG}(S_{n-1}, T)\neq 0$.
Set $V_{r+1}=T$ and $f(r)=s-1$, we have
$$\mathrm{Ext}_{kG}^{f(i)}(V_{i}, V_{i+1})\neq 0$$
where $f(i)$ are positive integers for $i=1,2,\ldots, r$.
So we complete the Case 2.1.

{\bf Case 2.2.} $\mathrm{Ker}(\kappa_0) \ncong S_{n-1}$.

Since $\mathrm{Ext}_{kG}^{f(r-1)}(V_{r-1}, V_{r})\neq 0$, set $t:=f(r-1)$, we have
 the following exact sequence: $$\CD
 0 @> >> V_r(=S_{n-1}) @> \gamma_{t} >> U_{t-1} @> \gamma_{t-1}  >> \cdots @> \gamma_1 >> U_0 @> \gamma_0 >> V_{r-1} @> >> 0
\endCD $$
$$\mathbf{Diagram ~5.2.}$$

So we have the following:
$$\CD
 0 @> >> S_{n-1} @> \gamma_t >> U_{t-1} @> \gamma_{t-1}  >> \mathrm{Im}(\gamma_{t-1})  @>  >> 0,
\endCD$$
$$\CD
 0 @> >> S_{n-1} @>\alpha_1 >> L @> \beta_1  >> T @> >> 0,
\endCD$$
and
$$\CD
0 @>  >> \mathrm{Ker}(\kappa_0) @> l >> M_0 @> \kappa_0 >> T @> >> 0
\endCD$$
where $l$ is an inclusion.
By Lemma 3.2,
we have
$$\xymatrix{
  ~ & ~ & ~ & 0 \ar[d]^{~}  & ~ &  ~ &  ~ \\
~ & ~ & 0 \ar[r]^{~} & S_{n-1}\ar[d]^{\alpha_1} \ar[r]^{\gamma_{t}} & U_{t-1} \ar[d]^{\tau_0} \ar[r]^{\gamma_{t-1}} & \mathrm{Im}(\gamma_{t-1}) \ar[r]^{~} & 0  \\
  ~ & ~ & W_2 \ar[d]^{\sigma_2} \ar[r]^{\sigma_1} & L\ar[d]^{\beta_1} \ar[r]^{\tau_1}  & W_1 &  ~ &  ~ \\
0\ar[r]^{~} & \mathrm{Ker}(\kappa_0) \ar[r]^{l} & M_0  \ar[r]^{\kappa_0} & T \ar[d]^{~} \ar[r]^{~} & 0 & ~ & ~  \\
  ~ & ~ & ~ & 0    & ~ &  ~ &  ~
} $$
$$\mathbf{Diagram ~5.3.}$$
where $W_1$ is the pushout and $W_2$ is the pullback which construct like Lemma 3.2.
That is, set $U=\{(\gamma_t(a), -\alpha_1(a))\in U_{t-1}\oplus L|a\in S_{n-1}\}$, we define $W_1:=(U_{t-1}\oplus L)/U.$ And
we have some morphisms as follows:
$$\tau_0: U_{t-1}\to W_1~ \mathrm{by}~\tau_0(x)=(x, 0)+U,$$
$$\tau_1: L\to W_1~ \mathrm{by}~\tau_1(y)=(0, y)+U$$
where $x\in U_{t-1}, y\in L.$
Set $$\gamma_{t-1}': W_1\to \mathrm{Im}(\gamma_{t-1})  ~\mathrm{by}~ \gamma_{t-1}'((x,y)+U)=\gamma_{t-1}(x).$$
So we have a short exact sequence as follows
$$\CD
 0 @> >> L @>\tau_1 >> W_1 @> \gamma_{t-1}' >> \mathrm{Im}(\gamma_{t-1}) @> >> 0
\endCD$$
by Lemma 3.2.

For $W_2$,  we can set $W_2=\{(y, z)\in L\oplus M_0|\beta_1(y)=\kappa_0(z)\}$ where $y\in L, z\in M_0$.  And
we have some morphisms as follows:
$$\sigma_1: W_2\to L ~\mathrm{by}~ \sigma_1(y, z)=y,$$
$$\sigma_2: W_2\to M_0 ~\mathrm{by}~ \sigma_2(y, z)=z$$
where $y\in L, z\in M_0$.
Set $l': \mathrm{Ker}(\kappa_0)\to W_2$ by $l'(b)=(0, l(b))$ for $b\in T$. So we have a short exact sequence as follows
$$\CD
 0 @> >> \mathrm{Ker}(\kappa_0) @>l' >> W_2 @> \sigma_1  >> L @> >> 0
\endCD$$
by Lemma 3.2.

Now, set $V_{r}':=\mathrm{Im}(\gamma_{t-1})$, $V_{r+1}:= L$, $V_{r+2}:= \mathrm{Ker}(\kappa_0)$ and $V_{r+3}:=T.$
By the Remark 3.3,
we can see that
$$\mathrm{Ext}_{kG}^{f(r-1)-1}(V_{r-1}, V_{r}')\neq 0;$$
$$\mathrm{Ext}_{kG}^{1}(V_{r}', V_{r+1})\neq 0;$$
$$\mathrm{Ext}_{kG}^{1}(V_{r+1}, V_{r+2})\neq 0;$$
$$\mathrm{Ext}_{kG}^{s-1}(V_{r+2}, V_{r+3})\neq 0.$$
Hence, we have a list $k=V_1, V_2, \ldots, V_{r-1}, V_r', V_{r+1}, V_{r+2}, V_{r+3}=T$ satisfies
$$\mathrm{Ext}_{kG}^{f(i)}(V_{i}, V_{i+1})\neq 0$$
where $f(i)$ are positive integers for $i=1,2,\ldots, r+2$.
So, we get the proof of this case.
\end{proof}

\textbf{ACKNOWLEDGMENTS}\hfil\break
The authors would like to thank Southern University of
Science and Technology for their kind hospitality hosting joint meetings of them.

\end{document}